\documentclass[12pt]{article}
\usepackage{amsmath}
\usepackage{amssymb}
\usepackage{braket} 
\usepackage{enumerate} 
\usepackage{amsthm}
\usepackage{graphicx}
\usepackage{psfrag}
\usepackage[T1]{fontenc}

\newtheorem{theorem}{Theorem}[section]
\newtheorem{proposition}[theorem]{Proposition}
\newtheorem{lemma}[theorem]{Lemma}
\newtheorem{corollary}[theorem]{Corollary}

\numberwithin{equation}{section}

\theoremstyle{definition}
\newtheorem{definition}[theorem]{Definition}
\newtheorem{remark}[theorem]{Remark}
\newtheorem{remarks}[theorem]{Remarks}
\newtheorem{question}[theorem]{Question}

\newcommand{\IIi}{${\rm{II}}_1\ $}
\newcommand{\tr}{\tau}

\newcommand{\nm}[1]{\left\|{#1}\right\|}
\newcommand{\VN}[1]{\mathcal L#1}
\newcommand{\Cover}[2]{K_{#1}\left(#2\right)}
\newcommand{\Pack}[2]{P_{#1}\left(#2\right)}

\newcommand{\FG}{\mathbb F}
\newcommand{\Mat}[1]{\mathbb M_{#1}(\mathbb C)}
\newcommand{\MatSA}[1]{\mathbb M_{#1}^{\textrm{sa}}(\mathbb C)}

\renewcommand{\P}{\mathcal P}
\newcommand{\Peq}{\mathcal P_{\textrm{eq}}}
\newcommand{\I}{\mathcal I}
\newcommand{\Gen}[1]{\mathcal G(#1)}
\newcommand{\GenSA}[1]{\mathcal G_{\textrm{sa}}(#1)}

\newcommand{\GenMin}[1]{\mathcal G^{\textrm{min}}(#1)}
\newcommand{\GenSAMin}[1]{\mathcal G_{\textrm{sa}}^{\textrm{min}}(#1)}
\newcommand{\Unitary}{\mathcal U}

\newcommand{\Fund}[1]{\mathcal F(#1)}
\newcommand{\ip}[2]{\left<#1,#2\right>}
\newcommand{\Tr}{\textrm{Tr}}
\newcommand{\ttr}{\textrm{tr}}

\newcommand{\refines}{\preceq}
\newcommand{\eps}{\varepsilon}
\newcommand{\TR}{\textrm{tr}}
\newcommand{\ko}{\mathcal K}

\title{Generators of \IIi factors}
\author{Ken Dykema\thanks{Research supported in part by NSF grant DMS-0600814}\\\normalsize\texttt{kdykema@math.tamu.edu}\and Allan Sinclair\\\normalsize\texttt{a.sinclair@ed.ac.uk}\and Roger Smith\thanks{Research supported in part by NSF grant DMS-0401043}\\\normalsize\texttt{rsmith@math.tamu.edu}\and Stuart White\\\normalsize\texttt{s.white@maths.gla.ac.uk}}
\date{}

\begin{document}
\maketitle
\begin{abstract}
In 2005, Junhao Shen introduced a new invariant, $\mathcal G(N)$, of a
diffuse von Neumann algebra $N$ with a fixed faithful trace, and he used
this invariant to give a unified approach to showing that large
classes of ${\mathrm{II}}_1$ factors $M$ are singly generated.  This paper
focuses on properties of this invariant. We relate $\mathcal G(M)$
to the number of self-adjoint generators of a ${\mathrm{II}}_1$
factor $M$: if $\mathcal G(M)<n/2$, then $M$ is generated by $n+1$
self-adjoint operators, whereas if $M$ is generated by $n+1$
self-adjoint operators, then $\mathcal G(M)\leq n/2$. The
invariant $\mathcal G(\cdot)$ is well-behaved under amplification,
satisfying $\mathcal G(M_t)=t^{-2}\mathcal G(M)$ for all $t>0$. In
particular, if $\mathcal G(\mathcal L\mathbb F_r)>0$ for any particular $r>1$, then the free group factors are pairwise non-isomorphic and are not singly generated for sufficiently large values of $r$. Estimates are given for forming free
products and passing to finite index subfactors and the basic construction.
We also examine a version of the invariant $\mathcal G_{\text{sa}}(M)$
defined only using self-adjoint operators; this is proved to satisfy
$\mathcal G_{\text{sa}}(M)=2\mathcal G(M)$.  Finally we give inequalities
relating a quantity involved in the calculation of $\mathcal
G(M)$ to the free-entropy dimension $\delta_0$ of a collection of generators for $M$.
\end{abstract}

\section{Introduction}

An old problem in von Neumann algebra theory is the question of whether each separable von Neumann algebra $N$ is singly generated. A single generator $x$ leads to two self-adjoint generators $\{x+x^*, i(x-x^*)\}$ and any pair $\{h,k\}$ of self-adjoint generators yields a single generator $h+ik$. Thus the single generation problem has an equivalent formulation as the existence of two self-adjoint generators. Earlier work in this area solved all cases except for the finite von Neumann algebras, \cite{Douglas.SingleGenerator,Pearcy.SingleGeneratorType1,Topping.Lectures,Wogen.Generators}. Here there has been progress in special situations, \cite{Ge.DecompositionProperties,Ge.GeneratorT,Shen.Generators}, but a general solution is still unavailable. Recently Junhao Shen, \cite{Shen.Generators}, introduced a numerical invariant $\mathcal{G}(N)$, and was able to show that single generation for ${\rm{II}}_1$ factors was a consequence of $\mathcal{G}(N)<1/4$. He proved that $\mathcal{G}(N)=0$ for various classes of \IIi factors, giving a unified approach to the single generation of \IIi factors with Cartan masa, with property $\Gamma$ and those factorising as tensor products of \IIi factors. His work settled some previously unknown cases as well as giving a unified approach to various situations that had been determined by diverse methods. It should be noted that 0 is the only value of Shen's invariant that is currently known. If strictly positive values are possible, then Corollary \ref{Scaling.NonSingle} guarantees examples of separable \IIi factors which are not singly generated.

In this paper, our purpose is to undertake a further investigation of this invariant, and to relate it to a quantity $\mathcal{G}^{\text{min}}(M)$ which counts the minimal number of generators for $M$. A related quantity $\mathcal{G}^{\text{min}}_{\text{sa}}(M)$ counts the minimal number of self-adjoint generators, and there is a parallel invariant $\mathcal{G}_{\text{sa}}(M)$ to $\mathcal{G}(M)$ which has a similar definition (given below) but which restricts attention to self-adjoint generating sets.

The contents of the paper are as follows. The second section gives the definitions of $\mathcal{G}(N)$ and $\mathcal{G}_{\text{sa}}(N)$ in terms of generating sets and finite decompositions of 1 as sums of orthogonal projections. This is a slightly different but equivalent formulation of the original one in \cite{Shen.Generators}. These are related by the inequalities $\mathcal{G}(N) \le \mathcal{G}_{\text{sa}}(N) \le 4\mathcal{G}(N)$, although it is shown subsequently that $\mathcal{G}(M) =2\mathcal{G}_{\text{sa}}(M)$ for all \IIi factors $M$. The main result of the third section is that the relation $\mathcal{G}(M) < n/2$ for \IIi factors $M$ implies generation by $n+1$ self-adjoint elements. The case $n=1$ is of particular interest since single generation is then a consequence of $\mathcal{G}(M)<1/2$.

The fourth section relates the generator invariant of a \IIi factor $M$ to that of a compression $pMp$. If $\tau(p)=t$, then $\mathcal{G}(pMp) = t^{-2}\mathcal{G}(M)$. Up to isomorphism, $M_t$ can be uniquely defined as $pMp$ for any projection $p\in M$ with $\tau(p)=t$, $0<t<1$. In a standard way, $M_t$ can be defined for any $t>0$ as $p(\mathbb{M}_n\otimes M)p$ where $n$ is any integer greater than $t$, and $p\in \mathbb{M}_n \otimes M$ is a projection of trace $t/n<1$. In this more general situation, the scaling formula $\mathcal{G}(M_t) = t^{-2}\mathcal{G}(M)$ for $t>0$ also holds. The subsequent section contains some consequences of the scaling formula, and also establishes it for the related invariant $\mathcal{G}_{\text{sa}}(M)$. This requires a more indirect argument since the method of passing between generating sets for $M$ and for $M_t$ does not preserve self-adjointness and so cannot apply to $\mathcal{G}_{\text{sa}}(M_t)$ although it is suitable for $\mathcal{G}(M_t)$. The equality $\mathcal{G}_{\text{sa}}(M) = 2\mathcal{G}(M)$ is also established in this section.

The sixth section is concerned with finite index inclusions $N\subseteq M$ of \IIi factors. The main results are that $\mathcal{G}(\langle M,e_N\rangle) \leq \mathcal{G}(M)$ and that $\mathcal{G}(N) = \lambda^2 \mathcal{G}(\langle M,e_N\rangle)$, where $\langle M,e_N\rangle$ is the basic construction and $\lambda$ denotes the index $[M:N]$. A standard result of subfactor theory is that $M$ is the basic construction $\langle N,e_P\rangle$ for an index $\lambda$ inclusion $P\subseteq N$, so two of these basic constructions scale $\mathcal{G}(\cdot)$ by $\lambda^2$. This suggests the formula $\mathcal{G}(\langle M,e_N\rangle ) = \lambda\mathcal{G}(M)$, but this is still an open problem.

Section \ref{Free} concentrates on free group factors and their generalisations, the interpolated free group factors. For $r\in (0,\infty]$, the formula $\mathcal{G}(\mathcal{L} \mathbb{F}_{1+r}) = r\alpha$ is established, where $\alpha$ is a fixed constant in the interval $[0,1/2]$. This leads to two possibilities, depending on the value of $\alpha$. If $\alpha=0$, then $\mathcal{L}(\mathbb{F}_{1+r})$ is singly generated for all $r>0$, while if $\alpha>0$, then the free group factors are pairwise non-isomorphic, being distinguished by the generator invariant. The paper concludes with a discussion of Voiculescu's modified free entropy dimension $\delta_0(X)$, where $X$ is a finite generating set for $M$. A quantity $\mathcal{I}(X)$ is introduced in the second section on the way to defining $\mathcal{G}(M)$. The main results of the last section are the inequalities $\delta_0(X) \le 1 +2{\mathcal I}(X)$ for general finite generating sets, and the stronger form $\delta_0(X) \le 1 + {\mathcal I}(X)$ for generating sets of self-adjoint elements. These have the potential for providing lower bounds for $\mathcal{G}(M)$.

Finally, a word on notation. For a subset $X$ of a von Neumann algebra $M$, $W^*(X)$ will denote the von Neumann algebra generated by $X$. It is not assumed that $W^*(X)$
automatically contains the identity of $M$. For example, $W^*(p)=\mathbb C p$ for a projection $p\in M$.

\section{The generator invariant}\label{Prelim}
The main focus of the paper is on \IIi factors. However we define the generator invariant and establish basic results in the context of diffuse finite von Neumann algebras with a fixed faithful  trace $\tr$, which is normalised with $\tr(1)=1$. Diffuse finite von Neumann algebras will be denoted by $N$, while $M$ is reserved for $\text{II}_1$ factors. 

\begin{definition}\label{Prelim.DefP}
Let $(N,\tr)$ be a finite von Neumann algebra with fixed trace.  Let $\P$ (or $\P(N)$ when the underlying algebra is unclear) denote the collection of all finite sets of mutually orthogonal projections in $N$ which sum to $1$.  An important subclass of $\P$ is the collection $\Peq$ of those $P$ which consist of projections of equal trace. Note that in the context of a finite factor $M$, $P=\{p_1,\dots,p_k\}\in\Peq(M)$ can be diagonalised in the sense that there exist matrix units $(e_{i,j})_{i,j=1}^k$ in $M$ satisfying $e_{i,i}=p_i$ for all $i$.  In this factor context,  the elements of $\Peq$ are referred to as the \emph{diagonalisable} elements of $\P$.
\end{definition}

\begin{definition}
Consider $P,Q\in\P$. Then $Q$ \emph{refines} $P$, written $Q\refines P$, when every $p\in P$ is a sum of elements of $Q$.  The sets $P$ and $Q$ are \emph{equivalent}, written $P\sim Q$, when there is a unitary $u\in N$ with $uPu^*=Q$.  Note that in a \IIi factor,  $P\sim Q$ if, and only if, the multiset of the traces of the elements of $P$ is the same as the multiset of the traces of the elements in $Q$.
\end{definition}

The ordering $Q\refines P$ defined above is chosen so that the map $P\mapsto\I(X;P)$ below is order preserving.  This will be established in Lemma \ref{Prelim.Refine}.

\begin{definition}\label{Prelim.DefI}
Let $N$ be a  finite von Neumann algebra.  Given $x\in N$ and $P\in\P$, define
$$
\I(x;P)=\sum_{\substack{p,q\in P\\pxq\neq 0}}\tr(p)\tr(q).
$$
For a finite subset $X\subset N$ and $P\in\P$, define 
$$
\I(X;P)=\sum_{x\in X}\I(x;P),
$$ 
and
\begin{equation}\label{Prelim.DefI.1}
\I(X)=\inf_{P\in\P}\I(X;P).
\end{equation}
On occasion, the notation $\I_N(X;P)$ and $\I_N(X)$ will emphasise the algebra $N$ and hence the choice of trace.
\end{definition}

The definition of $\I$ given above is formally different from that of Shen \cite[Definition 2.1]{Shen.Generators}, in that Shen only considers families from $\Peq$ and performs the limiting procedure in a slightly different order.  Nevertheless, the resulting invariant $\Gen{N}$ defined in Definition \ref{Prelim.DefG} below agrees with \cite[Definition 2.1]{Shen.Generators} in the case of diffuse von Neumann algebras. Before proceeding,  a few elementary observations are recorded.
\begin{remarks}\label{Prelim.RemarksI}
\begin{enumerate}
\item The inequality 
\[
0\leq\I(X)\leq\I(X;P)\leq|X|
\]
holds for all finite subsets $X$ in $N$ and all $P\in\P$.
\item If $X$ is a finite subset of $N$, then define $X^*=\Set{x^*|x\in X}$.  The equality
\[
\I(X^*;P)=\I(X;P)
\]
holds for all $P\in\P$.\label{Prelim.RemarksI.Star}
\item\label{Prelim.RemarksI.Commutant} Let $P\in\P$ be a set of $k$ projections. The estimate
$$
\I(x;P)=\sum_{\substack{p\in P\\px\neq 0}}\tr(p)^2\leq k\max\Set{\tr(p)^2|p\in\P}
$$
is valid for each $x\in N\cap P'$.
In particular, if $P\in\Peq$, then $$\I(x;P)\leq k^{-1}.$$
\item If $X=\{x_1,\dots,x_m\}\subset N$ and $P=\{p_1,\dots,p_k\}\in\Peq$\ then
$$
\I(X;P)=k^{-2}\left|\Set{(i,j,l)|p_ix_lp_j\neq 0}\right|.
$$
Writing $\I(X;P)$ in this way is often useful in calculations, an example being \cite[Theorem 4.1]{Shen.Generators}.
\item For any $z_1,z_2\in N$ and $P\in\P$, 
\begin{equation}\label{Prelim.RemarksI.1}
\I(z_1+z_2;P)\leq\I(z_1;P)+\I(z_2;P).
\end{equation}

\item If $X$ is a finite generating set of $N$, then so also is 
$$
Y=\Set{x+i x^*,i(x-i x^*)|x\in X}
$$ 
and, furthermore,
$$
\I(X;P)\leq\I(Y;P)\leq 4\I(X;P).
$$
This follows from the preceding remark. Take $z_1=x+i x^*$ and $z_2=x-i x^*$ for each $x\in X$ to obtain $\I(X;P)\leq\I(Y;P)$.  Now let $Y_1=\Set{x+i x^*|x\in X}$ and $Y_2=\Set{i(x-i x^*)|x\in X}$ so that (\ref{Prelim.RemarksI.1}) and item \ref{Prelim.RemarksI.Star} above combine to give 
\begin{equation}\label{Prelim.RemarksI.SA}
\I(Y_1;P)\leq 2\I(X;P) \quad \text{and}\quad \I(Y_2;P)\leq 2\I(X;P).
\end{equation}
\end{enumerate}
\end{remarks}

\begin{lemma}\label{Prelim.Refine}
Let $N$ be a  finite von Neumann algebra. Consider a finite subset $X\subset N$ and $P,Q\in\P$ with $Q\refines P$.  Then
$$
\I(X;Q)\leq\I(X;P).
$$
\end{lemma}
\begin{proof}
Take $x\in N$ and two pairs of orthogonal projections $(e_1,e_2)$ and $(f_1,f_2)$ in $N$.  If $(e_1+e_2)x(f_1+f_2)\neq0$, then
$$
\sum_{\substack{i,j\\e_ixf_j\neq 0}}\tr(e_i)\tr(f_j)\leq\tr(e_1+e_2)\tr(f_1+f_2).
$$
The result now follows by induction.
\end{proof}

In many applications it will be useful to know that the infimum defining $\Gen{X}$ in (\ref{Prelim.DefI.1}) can be taken through $\Peq$.  Recall that a separable diffuse abelian von Neumann algebra $A$ is isomorphic to $L^\infty[0,1]$ and, if $A$ is equipped with a trace, then  this isomorphism can be chosen so that the trace is given by $\int_0^1\cdot\ \mathrm{d}t$ on $L^\infty[0,1]$.
\begin{lemma}\label{Prelim.Diagonalisable}
Let $N$ be a diffuse von Neumann algebra and let $X$ be a finite subset of $N$.  For each $n\in\mathbb N$,
\begin{align*}
\I(X)&=\inf\Set{\I(X;P)| P\in\Peq,\quad |P|=nl\textrm{ some }l\in\mathbb N}\\
&=\lim_{l\rightarrow\infty}(\inf \Set{\I(X;P)| P\in\Peq,\quad |P|=nl}).
\end{align*}
\end{lemma}
\begin{proof}
Fix $\eps>0$ and find $Q=\{q_1,\dots,q_k\}\in\P$ with
\begin{equation}\label{Prelim.Diagonalisable.20}
\I(X;Q)\leq \I(X)+\eps.
\end{equation}
Let $A$ be a diffuse abelian subalgebra of $N$ with $Q\subset A$, and let
\begin{equation}\label{Prelim.Diagonalisable.21}
\delta=\frac{\eps}{2|X|k}.
\end{equation}
By approximating the family $\tr(q_1),\dots,\tr(q_k)$ by  rational numbers with common denominator $nl$, there exists $l_0\in\mathbb N$ such that if $l\geq l_0$, then there are $r_1,\dots,r_k\in\mathbb N$ with
$$
\frac{r_i}{nl}\leq\tr(q_i)\leq\frac{r_i}{nl}+\delta
$$
for each $i$.  Choose projections $p_1,\dots,p_k$ in $A$ with $0\leq p_i\leq q_i$ and $\tr(p_i)=r_i/nl$ for each $i$.  Let $p_{k+1}=1-\sum_{i=1}^kp_i$, so $\tr(p_{k+1})\leq k\delta$. Let $P_0=\{p_1,\dots,p_{k+1}\}\in\P$.   Thus $q_ixq_j\neq 0$ whenever $p_ixp_j\neq 0$.  Hence
\begin{equation}\label{Prelim.Diagonalisable.1}
\I(X;Q)=\sum_{x\in X}\sum_{\substack{1\leq i,j\leq k\\q_ixq_j\neq 0}}\tr(q_i)\tr(q_j)\geq\sum_{x\in X}\sum_{\substack{1\leq i,j\leq k\\p_ixp_j\neq 0}}\tr(p_i)\tr(p_j).
\end{equation}
Now 
\begin{align*}
\I(X;P_0)\leq&\ \sum_{x\in X}\sum_{\substack{1\leq i,j\leq k\\p_ixp_j\neq 0}}\tr(p_i)\tr(p_j)\\
&\quad+|X|\left(\sum_{i=1}^k\tr(p_i)\tr(p_{k+1})+\sum_{j=1}^k\tr(p_{k+1})\tr(p_j)+\tr(p_{k+1})^2\right)\\
\leq&\ \I(X;Q)+2|X|\tr(p_{k+1})\\
\leq&\ \I(X;Q)+2k|X|\delta\leq\I(X;Q)+\eps,
\end{align*}
by inequality (\ref{Prelim.Diagonalisable.1}), the estimate $\tr(p_{k+1})\leq k\delta$ and the choice of $\delta$ in (\ref{Prelim.Diagonalisable.21}). Finally, refine $P_0$ to find a family $P\refines P_0$ in $\Peq$ with $nl$ elements.  Lemma \ref{Prelim.Refine} gives
$$
\I(X)\leq\I(X;P)\leq \I(X;P_0)\leq\I(X;Q)+\eps<\I(X)+2\eps,
$$
by (\ref{Prelim.Diagonalisable.20}). Since this can be done for any $l\geq l_0$, and $\varepsilon>0$ was arbitrary, the result follows.
\end{proof}

These preliminaries allow the \emph{generator invariant} of a diffuse finite von Neumann algebra to be defined.

\begin{definition}\label{Prelim.DefG}
If $N$ is a finitely generated diffuse finite von Neumann algebra, define the \emph{generator invariant} $\Gen{N}$ by
$$
\Gen{N}=\inf\Set{\I({X})|W^*(X)=N,\quad X\textrm{ is a finite subset of }N},
$$
and the \emph{hermitian, (or self--adjoint), generator invariant}, $\GenSA{N}$ by
$$
\GenSA{N}=\inf\Set{\I({X})|W^*(X)=N,\quad X\textrm{ is a finite subset of  }N_{\text{sa}}}.
$$
If $N$ is not finitely generated, then define $\Gen{N}=\GenSA{N}=\infty$.
\end{definition}

\begin{remark}\label{Prelim.SelfRemark}
By item \ref{Prelim.RemarksI.SA} of Remarks \ref{Prelim.RemarksI}, the inequalities
$$
\Gen{N}\leq\GenSA{N}\leq 4\Gen{N}
$$
hold for any diffuse finite von Neumann algebra $N$.  In Theorem \ref{SelfAd.Main} it will be shown that $\Gen{M}=2\GenSA{M}$ for all \IIi factors $M$.  
\end{remark}

\section{The generation theorem}
Theorem 4.1 of \cite{Shen.Generators} states that if a \IIi factor $M$ has $\Gen{M}<1/4$, then it is generated by a projection and a hermitian element.  It is subsequently remarked that the same proof shows that a \IIi factor $M$ with $\Gen{M}<1/2$ is singly generated.  The theorem below strengthens this result to consider \IIi factors which may not be singly generated, with Shen's remark arising from the case $n=1$. The basic idea is the same combinatorical counting argument of \cite{Shen.Generators} which dates back through \cite{Ge.DecompositionProperties} and \cite{Ge.GeneratorT} to work in the  1960's  by Douglas, Pearcy and Wogen, \cite{Douglas.SingleGenerator,Pearcy.SingleGeneratorType1,Wogen.Generators}.  A good account of this material can be found in the book by Topping, \cite{Topping.Lectures}.

Recall that if $(e_{i,j})_{i,j=1}^k$ are matrix units for the $k\times k$ matrices, $\Mat{k}$, then the self--adjoint elements $\sum_{i=1}^{k-1}(e_{i,i+1}+e_{i+1,i})$ and $e_{k,k}$ generate $\Mat{k}$.

\begin{theorem}[The generation theorem]\label{Number.Hermitians}
Let $M$ be a separable $\text{\rm II}_1$ factor and $n\in\mathbb N$.  If $\Gen{M}<n/2$, then $M$ is generated by $n+1$ hermitian elements.
\end{theorem}
\begin{proof}
Suppose that $\Gen{M}<n/2$ for some $n\in\mathbb N$. There exists $k_0\in\mathbb N$ such that
$$
\Gen{M}<\frac{n}{2}-\left(\frac{n+2}{2k}-\frac{1}{k^2}\right)
$$
for all $k\geq k_0$. By Lemma \ref{Prelim.Diagonalisable}, there is a finite set $X=\{x_1,\dots,x_m\}$ of generators for $M$, some $k\geq k_0$ and a diagonalisable family $P=\{e_1,\dots,e_k\}\in\Peq$ such that
\begin{equation}\label{Number.Hermitians.1}
\I(X;P)<\frac{n}{2}-\left(\frac{n+2}{2k}-\frac{1}{k^2}\right).
\end{equation}
Choose a set of matrix units $(e_{i,j})_{i,j=1}^k$ in $M$ with $e_{i,i}=e_i$ for each $i$.  Define a set of triples
\begin{equation}\label{Number.Hermitians.20}
T=\Set{(i,j,l)|1\leq i,j\leq k,\quad 1\leq l\leq m,\quad e_ix_le_j\neq 0}
\end{equation}
so the definition of $\I(X;P)$ gives 
\begin{equation}\label{Number.Hermitians.3}
\I(X;P)=k^{-2}|T|.
\end{equation}

For each $r=1,\dots,n-1$, let $S_r$ be the set of triples $(s,t,r)$ with $1\leq s<t\leq k$ and let $S_n$ be the set of triples $(s,t,n)$ with $1\leq s<t\leq k-1$.  Then
$$
\left|\bigcup_{q=1}^nS_q\right|=\frac{(n-1)k(k-1)}{2}+\frac{(k-1)(k-2)}{2}=\frac{nk^2}{2}-\frac{(n+2)k}{2}+1.
$$
By the choice of $k\geq k_0$, 
$$
|T|<\frac{nk^2}{2}-\frac{(n+2)k}{2}+1=\left|\bigcup_{r=1}^{n}S_r\right|,
$$ 
from (\ref{Number.Hermitians.1}) and (\ref{Number.Hermitians.3}).

Decompose $T$ into a partition $\bigcup_{r=1}^{n}T_r$ with each $|T_r|\leq|S_r|$ and find, for each $r$, injections $T_r\rightarrow S_r$, written as $(i,j,l)\mapsto(s(i,j,l),t(i,j,l),r)$.  This really  defines $n$ maps indexed by $r$, but as the domains $T_r$ and ranges $S_r$ are disjoint, these are regarded as a single map from $T=\bigcup T_r$ to $\bigcup S_r$. Define self--adjoint operators
$$
y_r=\sum_{(i,j,l)\in T_r}(e_{s(i,j,l),i}x_le_{j,t(i,j,l)}+e_{t(i,j,l),j}x_l^*e_{i,s(i,j,l)})
$$
for $1\leq r\leq n$.

If $1\leq r\leq n$ and $(i_1,j_1,l_1)\in T_r\subset T$, then $(s(i_1,j_1,l_1),t(i_1,j_1,l_1))$ appears exactly once in the set
$$
\Set{(s,t),(t,s)|(s,t,r)\in S_r}
$$
as $S_r$ is disjoint from its transpose on the first two variables and the map $T_r\rightarrow S_r$ is an injection.  For such $(i_1,j_1,l_1)\in T_r$,
\begin{align}
e_{i_1,s(i_1,j_1,l_1)}y_re_{t(i_1,j_1,l_1),j_1}&=e_{i_1,s(i_1,j_1,l_1)}e_{s(i_1,j_1,l_1),i_1}x_{l_1}e_{j_1,t(i_1,j_1,l_1)}e_{t(i_1,j_1,l_1),j_1}\nonumber\\
&=e_{i_1}x_{l_1}e_{j_1}\neq 0.\label{Number.Hermitians.10}
\end{align}

As $T=\bigcup_{r=1}^nT_r$, equation (\ref{Number.Hermitians.10}) and the definition of $T$ in (\ref{Number.Hermitians.20}) imply that
\begin{align*}
x_l=\sum_{(i,j,l)\in T}e_{i,i}x_le_{j,j}=\sum_{r=1}^n\sum_{(i,j,l)\in T_r}e_{i,s(i,j,l)}y_re_{t(i,j,l),j)},
\end{align*}
for each $l=1,\dots, m$.  Thus the set $\{y_1,\dots,y_n\}\cup\Set{e_{i,j}|1\leq i,j\leq k}$  generates the \IIi factor $M$.

Finally note that $e_{k,k}y_n=y_ne_{k,k}=0$ so that  $y_n,e_{k,k}\in W^*(\{y_n+\lambda e_{k,k}\})$ for any $\lambda > \|y_n\|$ by the spectral theorem.  In this way $M$ is generated by the $n+1$ hermitian elements
$$
y_1,\dots,y_{n-1},\quad y_n+\lambda e_{k,k}, \text{ and } \sum_{i=1}^{k-1}(e_{i,i+1}+e_{i+1,i}),
$$
as required.
\end{proof}

\begin{remark}\label{Number.Remark}
In Corollary \ref{NGenerators}, it will be shown that if $M$ is generated by $n+1$ hermitian elements, then $\Gen{M}\leq n/2$, which almost gives a converse to this result.  The remaining gap is summarised in the following question.
\end{remark}

\begin{question}
Let $M$ be a \IIi factor.  If $\Gen{M}=n/2$ for some $n\in\mathbb N$, is $M$ generated by $n+1$ hermitians?
\end{question}

\section{A scaling formula}\label{Scaling}
This section examines the behaviour of $\Gen{M}$ under amplifications and  compressions. Theorem \ref{Scaling.Formula} establishes that
\begin{equation}\label{Scaling.1}
\Gen{M_t}=t^{-2}\Gen{M}
\end{equation}
for all \IIi factors $M$ and $t>0$.  Note that it is a consequence of equation (\ref{Scaling.1}) that $M$ is finitely generated if and only if $M_t$ is finitely generated for all $t>0$.  This result is deduced in Lemma \ref{Scaling.FiniteGen} from the lemmas that  will be needed to prove equation (\ref{Scaling.1}).

\begin{lemma}\label{Scaling.Second}
Let $M$ be a separable \IIi factor and let $n\in\mathbb N$.  Then
$$
\Gen{M_{n^{-1}}}\leq n^2\Gen{M}.
$$
\end{lemma}

\begin{proof}
Assume that $\Gen{M}<\infty$ as otherwise the lemma is vacuous.  Let $\eps>0$. By Lemma \ref{Prelim.Diagonalisable}, there is a finite generating set $X$ for $M$, some $k\in\mathbb N$ and a diagonalisable $P=\{e_1,\dots,e_{nk}\}\in\Peq$ such that $\I(X;P)<\Gen{M}+\eps$.  Find matrix units $(e_{i,j})_{i,j=1}^{nk}$ in $M$ with $e_{i,i}=e_i$.  Define 
$$
f_{r,s}=\sum_{i=1}^ke_{(r-1)k+i,(s-1)k+i}
$$
for $r,s=1,\dots,n$.  The family $(f_{r,s})_{r,s=1}^n$ is a system of matrix units in $M$.  Consider $f=f_{1,1}$, a projection of trace $n^{-1}$ so $fMf$ is a representative of $M_{n^{-1}}$.

The von Neumann algebra $fMf$ is generated by $\bigcup_{r,s=1}^nf_{1,r}Xf_{s,1}$ using induction on the equation
$$
fxyf=\sum_{r=1}^nf_{1,1}xf_{r,1}f_{1,r}yf_{1,1},
$$
see for example \cite[Lemma 5.2.1]{Dykema.FreeBook}.  Consider $Q=\{e_1,\dots,e_k\}$, a family of orthogonal projections in $fMf$ with sum $f$, so $Q$ is a diagonalisable element of $\Peq(fMf)$.  For $x\in X$, $i,j=1,\dots,k$ and $r,s=1,\dots,n$ the relation
$e_if_{1,r}xf_{s,1}e_j\neq 0$ is equivalent to $e_{(r-1)k+i}xe_{(s-1)k+j}\neq 0$
since 
$$
e_if_{1,r}xf_{s,1}e_j=e_{i,(r-1)k+i}xe_{(s-1)k+j,j}.
$$
Let $C$ be the number of quintuples $(i,j,r,s,x)$ with this property.

Since each projection in $Q$ has trace $k^{-1}$ in $fMf$, it follows that
$$
\I_{fMf}(\bigcup_{r,s=1}^nf_{1,r}Xf_{s,1};Q)=C\frac{1}{k^2},
$$
whereas
$$
\I_M(X;P)=C\frac{1}{n^2k^2}.
$$
Therefore
\begin{align*}
\Gen{M_{n^{-1}}}&=\Gen{fMf}\\&\leq \I_{fMf}\left(\bigcup_{r,s=1}^nf_{1,r}Xf_{s,1};Q\right)\\
&=n^2\I_M(X;P)\leq n^2\Gen{M}+n^2\eps,
\end{align*}
from which the lemma follows.
\end{proof}

\begin{remark}\label{Scaling.Second.Remark}
It does not appear to be possible to obtain 
\begin{equation}\label{Scaling.Second.Remark.1}
\GenSA{M_{n^{-1}}}\leq n^2\GenSA{M}
\end{equation}
using the methods of Lemma \ref{Scaling.Second}, since  a self--adjoint generating set $X$ leads to a generating set $\bigcup_{r,s=1}^nf_{1,r}Xf_{s,1}$  for $fMf$ which is not necessarily self--adjoint.  Nevertheless inequality (\ref{Scaling.Second.Remark.1}) is true, as will be established in Corollary \ref{SelfAd.Scaling} from the scaling formula of Theorem \ref{Scaling.Formula} and Theorem \ref{SelfAd.Main}.
\end{remark}

The next lemma provides one inequality in (\ref{Scaling.1}) and also the parallel inequality for the hermitian--generator invariant.
\begin{lemma}\label{Scaling.First}
Let $M$ be a separable \IIi factor and $0<t<1$.  Then
$$
\Gen{M}\leq t^2\Gen{M_t}\quad\textrm{and}\quad\GenSA{M}\leq t^2\GenSA{M_t}.
$$
\end{lemma}
\begin{proof}
Assume that $M_t$ is finitely generated, otherwise there is nothing to prove.  Fix a projection $p\in M$ of trace $t$ so that $pMp$ is a representative of $M_t$ and let $X$ be an arbitrary finite set of generators for $pMp$.

For $\eps>0$, find orthogonal projections $E=\{e_1,\dots,e_n\}\in\P_{pMp}$ such that
\begin{equation}\label{Scaling.First.1}
\I_{pMp}(X;E)=t^{-2}\sum_{x\in X}\sum_{e_ixe_j\neq 0}\tr_M(e_i)\tr_M(e_j)<\I_{pMp}(X)+\eps.
\end{equation}
The factor $t^{-2}$ arises in (\ref{Scaling.First.1}) as $\tr_{M_t}(y)=t^{-1}\tr_{M}(y)$ for $y\in M_t$. 
Let $m$ be the maximal integer such that $mt\leq1$ and find a family of orthogonal projections $p_1,\dots,p_{m+1}$ in $M$ such that:
\begin{enumerate}[i.]
\item $p_1=p$;
\item $\tr(p_i)=\tr(p)$, for $i=2,\dots,m$;
\item $\sum_{i=1}^{m+1}p_i=1$.
\end{enumerate}
In this way $0\leq\tr(p_{m+1})<\tr(p)$ with $p_{m+1}$ possibly the zero projection.  Let $v_1=p_1$ and find partial isometries $v_2,\dots,v_{m+1}\in M$ such that: 
\begin{enumerate}
\item $v_iv_i^*=p_i$, for each $i=2,\dots,m+1$;
\item $v_i^*v_i=p_1$, for $i=2,\dots,m$;
\item $v_{m+1}^*v_{m+1}$ is a subprojection of $p_1$ of the form $\sum_{j=1}^{k-1}e_j+\widetilde{e_{k+1}}$ for some $k\in \{1,\ldots,n\}$ and some $\widetilde{e_k}\leq e_k$.
\end{enumerate}
It will now be shown that
\begin{equation}\label{Scaling.First.10}
Y=X\cup\{v_2,\dots,v_{m+1}\}
\end{equation}
generates the \IIi factor $M$. Since $X$ generates $pMp$, it follows that $pMp\subset
W^*(Y)$. The relations
$$
p_iMp_j=v_iv_i^*Mv_jv_j^*=v_ipv_i^*Mv_jpv_j^*,
$$
for $1\leq i,j\leq m+1$, then imply that
$$
p_iMp_j\subset v_ipMpv_j^*\subset v_iW^*(Y)v_j^*\subset
W^*(Y),
$$
and so $Y$ generates $M$.

The final step is to estimate $\I(Y)$. Let $F$ be a set of mutually orthogonal projections with sum $p$ that refines $\{e_1,\dots,e_{k-1},\widetilde{e_k},e_k-\widetilde{e_k},e_{k+1},\dots,e_n\}$ such that
\begin{equation}\label{nowneeded}
\max_{f\in F}\tr(f)<\frac{\eps}{m}.
\end{equation}
Thus $F\in\P_{pMp}$ and $F\refines E$.  Extend $F$ to a family
$$
G=\Set{v_ifv_i^*|1\leq i\leq m+1,\ f\in F}
$$
of projections.  Then $G\supset F$ and the sum of the orthogonal projections in $G$ is $1$ so $G\in\P_M$.  Fix $i\in \{2,\dots,m+1\}$ and note that if $v_jf_1v_j^*v_iv_kf_2v_k^*\neq 0$ for some $j,k$ and $f_1,f_2\in F$, then $j=i$ (as otherwise $v_j^*v_i=0$) and  $k=1$ (as otherwise $v_iv_k=0$); when these conditions hold $f_1=f_2$ (as $f_1v_j^*v_iv_kf_2=f_1p_1f_2=f_1f_2$).  Thus
\begin{equation}\label{Scaling.First.2}
\I_M(v_i;G)\leq \sum_{f\in F}\tr(f)^2\leq \max_{f\in F}\tr(f)<\frac{\eps}{m},
\end{equation}
and Lemma \ref{Prelim.Refine} and (\ref{Scaling.First.1}) imply that
\begin{align}\label{Scaling.First.12}
\I_M(X;G)=\I_M(X;F)&\leq\I_M(X;E)\nonumber\\
&=\sum_{\substack{e_1,e_2\in E,x\in X\\e_1xe_2\neq 0}}\tr(e_1)\tr(e_2)<t^2\I_{pMp}(X)+\eps t^2.
\end{align}
Finally
\begin{align}
\mathcal G(M)&\leq \I_M(Y;G)=\I_M(X;F)+\sum_{i=2}^{m+1}\I_M(v_i;G)\notag \\
\label{Scaling.First.12a} &\leq\I_M(X;E)+\eps<t^2\I_{pMp}(X)+\eps t^2+\eps.
\end{align}
The inequality $\Gen{M}\leq t^2\Gen{M_t}$ then follows, as $\eps >0$ was arbitrary and $X$ was an arbitrary finite generating set for $pMp$.

Only minor modifications are necessary for the hermitian case. Assume that each element of $X$ is self-adjoint, and replace the  partial isometries $v_i$ in \eqref{Scaling.First.10} by the self-adjoint elements $v_i+v_i^*$ and  $i(v_i-v_i^*)$. Then the same arguments lead to the inequality
\begin{align*}
\mathcal G_{{\text{sa}}}(M)<t^2\I_{pMp}(X)+\eps t^2+4\eps,
\end{align*}
for any finite generating set $X$ consisting of self-adjoint elements. The factor of 4, which was not present in the counterpart inequality \eqref{Scaling.First.12a}, arises from the estimate
$$
\sum_{i=2}^{m+1}\I_M(v_i+v_i^*;G)+\sum_{i=2}^{m+1}\I_M(i(v_i-v_i^*);G)\leq 4\sum_{i=2}^{m+1}\I_M(v_i;G)<4\eps.
$$
This change reflects the replacement of $v_i$ by $v_i+v_i^*$ and $i(v_i-v_i^*)$ in the generating set $Y$ of \eqref{Scaling.First.10}.
\end{proof}

\begin{lemma}\label{Scaling.FiniteGen}
Let $M$ be a \IIi factor and let $t>0$.  Then $M_t$ is finitely generated if, and only if, $M$ is finitely generated.
\end{lemma}
\begin{proof}
Without loss of generality, suppose that $0<t<1$, otherwise replace $t$ by $t^{-1}$.  Lemma \ref{Scaling.First} shows that if $M_t$ is finitely generated so too is $M$.  Conversely, if $M$ is finitely generated, choose $n\in\mathbb N$ with $n^{-1}<t$.  Lemma \ref{Scaling.Second} shows that $M_{n^{-1}}$ is finitely generated. Let $s=n^{-1}t^{-1}$.  Since $0<s<1$ and $(M_t)_s=M_{n^{-1}}$ is finitely generated, another application of Lemma \ref{Scaling.First} shows that $M_t$ is finitely generated.
\end{proof}

\begin{theorem}[The Scaling Formula]\label{Scaling.Formula}
Let $M$ be a separable \IIi factor.  For each $t>0$, 
\begin{equation}\label{Scaling.Formula.1}
\Gen{M_t}=t^{-2}\Gen{M}.
\end{equation}
\end{theorem}
\begin{proof}
This formula will be established by considering successively the cases $t\in \mathbb{Q}$, $t\in (0,1)$ and $t\in (1,\infty)$. 

By Lemma \ref{Scaling.FiniteGen}, $M$ is infinitely generated if, and only if, $M_t$ is infinitely generated.  Assume then that both $M$ and $M_t$ are finitely generated. For $n\in\mathbb N$ and any separable \IIi factor $M$, the equation
\begin{equation}\label{Scaling.Formula.2}
\Gen{M_{n^{-1}}}=n^2\Gen{M}
\end{equation}
is a consequence of the inequalities of Lemma \ref{Scaling.Second} and Lemma  \ref{Scaling.First}.  Let $t=p/q$ be a rational and apply (\ref{Scaling.Formula.2}) twice.  This gives
$$
\Gen{M_{t^{-1}}}=\Gen{(M_q)_{p^{-1}}}=p^2\Gen{M_q}
$$
and
$$
\Gen{M}=\Gen{(M_q)_{q^{-1}}}=q^2\Gen{M_q}=\frac{q^2}{p^2}\Gen{M_{t^{-1}}}.
$$
This proves the theorem for rational $t$.

For arbitrary $0<t<1$, consider $0<s<1$ such that $st\in\mathbb Q$. Then
\begin{equation}\label{Scaling.Formula.3}
\Gen{M}\leq t^2\Gen{M_t}\leq s^2t^2\Gen{M_{st}}=\Gen{M}
\end{equation}
by applying Lemma \ref{Scaling.First} twice and the rational case above.  Hence $\Gen{M}=t^2\Gen{M_t}$. 

If $t>1$, then 
$$
\Gen{M}=\Gen{(M_t)_{t^{-1}}}=t^2\Gen{(M_t)_{t^{-1}t}}=t^2\Gen{M_t}.
$$
This deals with all possible cases and so completes the proof.
\end{proof}

\section{Consequences of scaling}
This section contains an initial collection of deductions from the scaling formula. Results regarding the free group factors and estimates involving free products are reserved to Section \ref{Free}. Note that there is currently no  example for which the hypothesis of the first corollary is known to hold. 
\begin{corollary}\label{Scaling.NonSingle}
If there exists a separable \IIi factor $M$ such that $\Gen{M}>0$, then there exist separable \IIi factors which are not singly generated.
\end{corollary}
\begin{proof}
If the separable \IIi factor $M$ satisfies $\Gen{M}=\infty$, then this is already an example with no finite set of generators. Thus the additional assumption that $0<\Gen{M}<\infty$ can be made, in which case 
$$
\lim_{t\to 0+}\Gen{M_t}=\lim_{t\to 0+}t^{-2}\Gen{M}=\infty ,
$$
by Theorem \ref{Scaling.Formula}. It is immediate from the definition of $\Gen{\cdot}$ that any singly generated factor has $\Gen{\cdot}\leq 1$, and so $M_t$ is not singly generated for sufficiently small values of $t$.
\end{proof}

\begin{corollary}\label{Scaling.FundamentalGroup}
Any finitely generated \IIi factor $M$ with non--trivial fundamental group $\Fund{M}$ has $\Gen{M}=0$ and is singly generated.
\end{corollary}
\begin{proof}
If $t\in\Fund{M}\setminus\{1\}$, then $\Gen{M}=\Gen{M_t}=t^{-2}\Gen{M}$ by Theorem \ref{Scaling.Formula} as $M\cong M_t$.  Since $\Gen{M}<\infty$, it follows that $\Gen{M}=0$ so  $M$ is singly generated by Theorem \ref{Number.Hermitians} or \cite{Shen.Generators}.
\end{proof}

The following lemma leads to the upper bound for $\Gen{M}$ in Remark \ref{Number.Remark}, and which will finally be established in Corollary \ref{NGenerators}.  A simple modification yields the odd integer case of Corollary \ref{NGenerators} without the need for further work.
\begin{lemma}\label{SelfAd.NGenerators}
Let $M$ be a \IIi factor which is generated by $n$ hermitian elements for some $n\in\mathbb N$.  Then $\GenSA{M}\leq n-1$.
\end{lemma}
\begin{proof}
Let $X=\{x_1,\dots,x_n\}$ be $n$ hermitian elements generating $N$.  Let $\eps>0$ and let $A$ be a masa containing $x_n$.  Since $A\cong L^\infty[0,1]$, there exists some diagonalisable $P\in\Peq(A)$ with $|P|>\eps^{-1}$.  Then  $\I(x_n;P)\leq\eps$ from item \ref{Prelim.RemarksI.Commutant} of Remarks \ref{Prelim.RemarksI}.  For each $i\in \{1,\dots,n-1\}$,  the  estimate $\I(x_i;P)\leq 1$ gives
$$
\GenSA{M}\leq\I(X;P)\leq n-1+\eps
$$
and the result follows.
\end{proof}

The next lemma is the easy direction of Theorem \ref{SelfAd.Main}.
\begin{lemma}\label{SelfAd.Easy}
Let $M$ be a \IIi factor.  Then
$$
2\Gen{M}\leq\GenSA{M}.
$$
\end{lemma}
\begin{proof}
Suppose that $M$ is finitely generated, as otherwise the inequality is vacuous.  Let $\eps>0$. Use Lemma \ref{Prelim.Diagonalisable} to find a self--adjoint set of generators $X=\{x_1,\dots,x_n\}$ for $M$ and a diagonalisable set of projections $P\in\Peq$ with $\I(X;P)\leq\GenSA{M}+\eps$.  Take $k=|P|m>n/\eps$ for some $m\in\mathbb N$ and choose a diagonalisable refinement $Q=\{e_1,\dots,e_k\}\in\Peq$ of $P$.  For $l\in \{1,\dots,n\}$ let
$$
y_l=\sum_{1\leq i< j\leq k}e_ix_le_j,\quad\textrm{and}\quad z_l=\sum_{i=1}^ke_ix_le_i.
$$
Then let $Y=\{y_1,\dots,y_n\}$ and $Z=\{z_1,\dots,z_n\}$.

Since $x_l=z_l+y_l+y_l^*$ for each $l$, the finite set $Y\cup Z$ generates $M$.  Moreover,
$$
2\I(Y;Q)+\I(Z;Q)=\I(Y;Q)+\I(Y^*;Q)+\I(Z;Q)=\I(X;Q)
$$
since the regions in $\{1,\dots,k\}\times\{1,\dots,k\}$ depending on $Y$, $Y^*$ and $Z$ are disjoint. This gives the estimate
$$
\I(Y;Q)\leq\I(Y;Q)+\frac{1}{2}\I(Z;Q)\leq\frac{1}{2}\I(X;Q).
$$

As $z_l\in Q'$ for all $l$, item \ref{Prelim.RemarksI.Commutant} of Remarks \ref{Prelim.RemarksI} gives $\I(Z;Q)\leq n/k<\eps$.  Hence
$$
\Gen{M}\leq\I(Y;Q)+\I(Z;Q)\leq\frac{1}{2}\I(X;Q)+\eps\leq \frac{1}{2}\GenSA{M}+\frac{3}{2}\eps
$$
as $Q\refines P$.  The lemma follows.
\end{proof}

The generator invariant can now be related to the hermtian generator invariant.
\begin{theorem}\label{SelfAd.Main}
Let $M$ be a \IIi factor. Then
$$
\Gen{M}=\frac{1}{2}\GenSA{M}.
$$
\end{theorem}
\begin{proof}
By Lemma \ref{SelfAd.Easy}, it suffices to prove that $\GenSA{M}\leq 2\Gen{M}$ for a finitely generated \IIi factor $M$.  Let $\eps>0$ and choose $k,n\in\mathbb N$ with $k>1$ such that
$$
\Gen{M}<\frac{n}{2k^2}\leq\Gen{M}+\eps.
$$
Write $t=1/k$ so that $0<t<1$.  The scaling formula (Theorem \ref{Scaling.Formula}) gives
$$
\Gen{M_t}=t^{-2}\Gen{M}<\frac{n}{2}.
$$
By the generation theorem (Theorem \ref{Number.Hermitians}), $M_t$ is generated by $n+1$ hermitian elements so that $\GenSA{M}\leq n$ by Lemma \ref{SelfAd.NGenerators}. Lemma \ref{Scaling.First} on scaling by small $t$ implies that
$$
\GenSA{M}\leq t^2\GenSA{M_t}\leq t^2n=\frac{n}{k^2}<2\Gen{M}+2\eps.
$$
This proves the theorem.
\end{proof}

The scaling for the  generator invariant can now be extended to the hermitian case, as stated in Remark \ref{Scaling.Second.Remark}.
\begin{corollary}\label{SelfAd.Scaling}
If $M$ is a separable \IIi factor and $t>0$, then 
$$
\GenSA{M_t}=t^{-2}\GenSA{M}.
$$
\end{corollary}
\begin{proof}
This follows directly from Theorems \ref{Scaling.Formula} and \ref{SelfAd.Main}.
\end{proof}

The next corollary gives the partial converse to the generation theorem which was indicated in Remark \ref{Number.Remark}. 
\begin{corollary}\label{NGenerators}
If $n$ is the minimal number of hermitian generators of a finitely generated \IIi factor $M$, then
$$
n-1\leq 2\Gen{M}+1\leq n.
$$
\end{corollary}
\begin{proof}
Lemma \ref{SelfAd.NGenerators} and Theorem \ref{SelfAd.Main} give $\Gen{M}\leq (n-1)/2$ which yields the second inequality. If $2\Gen{M}+1<n-1$, then $\Gen{M}<\frac{n-2}{2}$ so that the generation theorem (Theorem \ref{Number.Hermitians}) gives $n-1$ hermitian generators for $M$, contradicting the minimality of $n$.  This gives the first inequality.
\end{proof}

\begin{corollary}\label{NGenerators.Group}
If $G$ is a countable discrete I.C.C. group generated by $n$ elements, then $\Gen{\VN{G}}\leq (n-1)/2$.  
\end{corollary}
\begin{proof}
If $G$ is generated by $g_1,\dots,g_n$, then there are hermitian elements $h_i$ in the \IIi factor $\VN{G}$ with $W^*(h_i)=W^*(g_i)$ for all $i$ by the spectral theorem, since each $g_i$ is normal.  The previous corollary completes the proof.
\end{proof}

The following result can be found in \cite{Ge.DecompositionProperties}.  It dates back to Douglas and Percy \cite{Douglas.SingleGenerator} and a proof can also be found in \cite[Chapter 15]{Sinclair.MasaBook}.
\begin{quote}
\emph{If a von Neumann algebra $N$ is generated by $k$ self--adjoint elements, then $\Mat{k}\otimes N$ is generated by two self--adjoint elements, equivalent to being singly generated.}
\end{quote}

In the case of \IIi factors, this can be strengthened.
\begin{corollary}\label{Scaling.Matrices}
If $M$ is a \IIi factor generated by $k^2$ self--adjoint elements, then $M\otimes\mathbb M_k(\mathbb C)$ is generated by two self--adjoint elements, equivalent to being singly generated.
\end{corollary}
\begin{proof}
By Corollary \ref{NGenerators}
$$
\Gen{M}\leq\frac{k^2-1}{2}.
$$
The scaling formula (Theorem \ref{Scaling.Formula}) gives
$$
\Gen{M\otimes\Mat{k}}=\frac{1}{k^2}\Gen{M}\leq\frac{k^2-1}{2k^2}<1/2.
$$
The result follows from the generation theorem (Theorem \ref{Number.Hermitians}) or \cite{Shen.Generators}.
\end{proof}

Corollary \ref{MinGen} gives a formula describing the generator invariant in terms of the minimal number of generators required to generate compressions of the von Neumann algebra, based on the following:  
\begin{definition}
Let $M$ be a \IIi factor.  Write $\GenMin{M}$ for the minimal number of generators of $M$ if $M$ is finitely generated and let $\GenMin{M}=\infty$ if $M$ is not finitely generated.  The quantity $\GenSAMin{M}$ has a similar definition in terms of the number of self-adjoint generators.
\end{definition}

The following simple estimates will be used below.
\begin{enumerate}
\item $\Gen{M}\leq \GenMin{M}$ and $\GenSA{M}\leq\GenSAMin{M}$.\label{MinGen.Easy.1}
\item $\GenSAMin{M}\leq 2\GenMin{M}\leq\GenSAMin{M}+1$.\label{MinGen.Easy.2}
\end{enumerate}

\begin{corollary}\label{MinGen}
Let $M$ be a separable \IIi factor.  Then
\begin{equation}\label{MinGen.1}
\Gen{M}=\lim_{k\rightarrow\infty}\frac{\GenMin{M_{1/k}}}{k^2}
\end{equation}
and
\begin{equation}\label{MinGen.2}
\GenSA{M}=\lim_{k\rightarrow\infty}\frac{\GenSAMin{M_{1/k}}}{k^2}.
\end{equation}
\end{corollary}
\begin{proof}
By Lemma \ref{Scaling.FiniteGen},  assume that $M$ is finitely generated.  The first step is to establish (\ref{MinGen.1}).  The scaling formula (Theorem \ref{Scaling.Formula}) and the estimate \ref{MinGen.Easy.1} above give
$$
\Gen{M}=\frac{\Gen{M_{1/k}}}{k^2}\leq\frac{\GenMin{M_{1/k}}}{k^2}
$$
for all $k$.  For $\eps>0$ find $k_0\in\mathbb N$ with $k_0^2>\eps^{-1}$.  For $k\geq k_0$, find $n\in\mathbb N$ with
$$
2k^2\Gen{M}<n\leq 2k^2\Gen{M}+1.
$$
The scaling formula gives 
\[
\Gen{M_{1/k}}=k^2\Gen{M}<n/2
\]
so the generation theorem (Theorem \ref{Number.Hermitians}) ensures that $M_{1/k}$ is generated by $n+1$ hermitians.  Estimate \ref{MinGen.Easy.2} preceding the corollary gives
$$
\GenMin{M_{1/k}}\leq\frac{1}{2}\left(\GenSAMin{M_{1/k}}+1\right)\leq\frac{n}{2}+1
$$
so that
$$
\frac{\GenMin{M_{1/k}}}{k^2}\leq \frac{n}{2k^2}+\frac{1}{k^2}\leq\Gen{M}+\frac{2}{k^2}\leq\Gen{M}+2\eps.
$$
This gives  equation (\ref{MinGen.1}).  For (\ref{MinGen.2}), note that the estimate \ref{MinGen.Easy.2} above gives
$$
\lim_{k\rightarrow\infty}\frac{\GenSAMin{M_{1/k}}}{k^2}=2\lim_{k\rightarrow\infty}\frac{\GenMin{M_{1/k}}}{k^2}.
$$
The result follows from (\ref{MinGen.1}) and Theorem \ref{SelfAd.Main}.
\end{proof}

The last result of this section records that the previous corollary is part of a general phenomenon regarding invariants of \IIi factors which scale and are bounded by the numbers of generators involved.  The proof is omitted.
\begin{corollary}
Suppose that $\mathcal H$ is an invariant of a \IIi factor which satisfies $\mathcal H(M)=t^2\mathcal H(M_t)$ for all $t>0$. If there exist constants $a,b\geq 0$ and $\alpha,\beta\in\mathbb R$ with 
$$
a\GenMin{M}+\alpha\leq\mathcal H(M)\leq b\GenMin{M}+\beta
$$
for all separable \IIi factors $M$, then
$$
a\Gen{M}\leq\mathcal H(M)\leq b\Gen{M}.
$$
\end{corollary}

\section{Finite index subfactors}
This brief section examines the generator invariant for finite index inclusions of \IIi factors.  Recall from \cite{Jones.Index} that if $N\subset M$ is a unital inclusion of \IIi factors and $e_N$ is the orthogonal projection from $L^2(M)$ onto $L^2(N)$, then the basic construction $\ip{M}{e_N}$ is the von Neumann subalgebra of $\mathbb B(L^2(M))$ generated by $M$ and $e_N$ and $\ip{M}{e_N}=JN'J$, where $J$ is the usual modular conjugation operator on $L^2(M)$ given by extending the map $x\mapsto x^*$. This last equation holds for any von Neumann subalgebra of a \IIi factor $M$, and implies that  $\ip{M}{e_N}$ is a factor precisely when the same is true for $N$.  Recall also that $\{e_N\}'\cap\ip{M}{e_N}=N$.  In this situation, $N$ is said to be a \emph{finite index} subfactor if $\ip{M}{e_N}$ is a type \IIi factor, and one formulation of the index $[M:N]$ is given by $[M:N]=\Tr(1)$, where $\Tr$ is the unique trace on $\ip{M}{e_N}$ normalised with $\Tr(e_N)=1$.
\begin{lemma}\label{Subfactor.BCon}
Suppose that  $N\subset M$ is a unital inclusion of \IIi factors with $[M:N]<\infty$.  Then 
$$
\Gen{\ip{M}{e_N}}\leq\Gen{M}.
$$
\end{lemma}
Note that in this lemma, $\Gen{\ip{M}{e_N}}$ is computed with respect to the trace $\lambda^{-1}\Tr$, where $\lambda=[M:N]$, since this is the trace on the basic construction algebra $\ip{M}{e_N}$, normalised to take the value $1$ at the identity.
\begin{proof}
Take $\eps>0$, a finite generating set $X$ for $M$ and (by Lemma \ref{Prelim.Diagonalisable}) a collection of projections $P\in\Peq(M)$ with $\I(X;P)<\Gen{M}+\eps$.  Since $N$  is also a factor, there exists a unitary $u\in\Unitary(N)$ such that $uPu^*=Q_0\in\P(N)$.  Choose a refinement $Q$ of $Q_0$ in $\Peq(N)$ with $|Q|=k$ for some $k>\eps^{-1}$.    The inequality $$\I(e_N;Q)\leq k^{-1}<\varepsilon$$ is implied by item \ref{Prelim.RemarksI.Commutant} of Remarks \ref{Prelim.RemarksI}. Since $uXu^*$ also generates $M$, it follows that
$$
\Gen{\ip{M}{e_N}}\leq\I(uXu^*\cup\{e_N\};Q)\leq\I(X;P)+\eps<\Gen{M}+2\eps,
$$
proving the result.
\end{proof}

\begin{lemma}\label{Subfactor.Subfactor}
Let $N\subset M$ be a finite index unital inclusion of \IIi factors.  Then
$$
\Gen{N}\geq\Gen{M}.
$$
\end{lemma}
\begin{proof}
Recall from \cite[Lemma 3.1.8]{Jones.Index}, that given a finite index inclusion $N\subset M$, there exists a subfactor $P\subset N$ with $[N:P]=[M:N]$ and $\ip{N}{e_P}\cong M$.  The result follows immediately from the previous lemma.
\end{proof}

Note that if there is a \IIi factor $M$ with $\Gen{M}>0$, then the previous lemma does not hold for infinite index subfactors as there is always a copy of the hyperfinite \IIi factor $R$ inside any \IIi factor.

\begin{theorem}
Let $N\subset M$ be a finite index unital inclusion of \IIi factors. Then $\Gen{M}=0$ if, and only if, $\Gen{N}=0$.
\end{theorem}
\begin{proof}
Suppose that $\Gen{M}=0$.  By Lemma \ref{Subfactor.BCon}, it follows that $\Gen{\ip{M}{e_N}}=0$.  Now $N\cong e_N\ip{M}{e_N}e_N$, so $N\cong\ip{M}{e_N}_{\lambda^{-1}}$, where $\lambda=[M:N]$.  The scaling formula (Theorem \ref{Scaling.Formula}) immediately gives $\Gen{N}=0$.  The reverse direction is Lemma \ref{Subfactor.Subfactor} above.
\end{proof}

More generally, suppose that $N\subset M$ is a finite index unital inclusion of \IIi factors and write $\lambda=[M:N]$.  The isomorphism $N\cong e_N\ip{M}{e_N}e_N$ and the scaling formula lead to the equality
$$
\Gen{N}=\lambda^2\Gen{\ip{M}{e_N}}.
$$
Furthermore, Lemmas \ref{Subfactor.BCon} and \ref{Subfactor.Subfactor} give
$$
\Gen{N}\geq\Gen{M}\geq\Gen{\ip{M}{e_N}}=\lambda^{-2}\Gen{N}.
$$
It is then natural to pose the following question.
\begin{question}
Suppose that $N\subset M$ is a finite index unital inclusion of \IIi factors and write $\lambda=[M:N]$.  Is it the case that
$$
\Gen{N}=\lambda\Gen{M}?
$$
\end{question}

\section{Free group factors and free products}\label{Free}
The main result in this section is Theorem \ref{Scaling.FreeGroup} below, which is obtained by using the quadratic scaling of the generator invariant and
a similar property of the interpolated free group factors $\VN{\FG_{s}}$ of the first author and R\u adulescu \cite{Dykema.InterpolatedFreeGroup,Radulescu.MatricesFreeProd}.   For $r>1$ and $\lambda>0$, the quadratic scaling formula for the interpolated free group factors states that
\begin{equation}\label{Scaling.FreeGroup.10}
(\VN{\FG_r})_\lambda=\VN{\FG_{1+\frac{r-1}{\lambda^2}}},
\end{equation}
from \cite[Theorem 2.4]{Dykema.InterpolatedFreeGroup}. 

\begin{theorem}\label{Scaling.FreeGroup}
There exists a constant $0\leq\alpha\leq 1/2$ such that
\begin{equation}\label{Scaling.FreeGroup.2}
\Gen{\VN{\FG_{1+r}}}=r\alpha.
\end{equation}
for all $r\in(0,\infty]$.
\end{theorem}
\begin{proof}
For $r\in(0,\infty)$ write $\beta(r)=\Gen{\VN{\FG_{r+1}}}$.   Equation (\ref{Scaling.FreeGroup.10}) above and Theorem \ref{Scaling.Formula} combine to give
\begin{equation}\label{Scaling.FreeGroup.1}
\beta\left(\frac{r-1}{\lambda^2}\right)=\Gen{(\VN{\FG_r})_\lambda}=\lambda^{-2}\Gen{\VN{\FG_r}}=\lambda^{-2}\beta(r-1).
\end{equation}
Take $r-1=t$ and $\lambda^{-2}=s$ in (\ref{Scaling.FreeGroup.1}) to obtain 
$$
\beta(st)=s\beta(t)
$$
for all $s,t\in(0,\infty)$. Hence, there is a constant $\alpha=\beta(1)\geq 0$ with $\beta(t)=\alpha t$.  The estimate $\alpha\leq\frac{1}{2}$ follows from Corollary \ref{NGenerators.Group} as $\alpha={\mathcal G}(\VN{\FG_2})$ and $\FG_2$ is certainly generated by two elements.

It remains to extend (\ref{Scaling.FreeGroup.2}) to the $r=\infty$ case. Since the fundamental group of $\VN{\FG_\infty}$ is $\mathbb R^+$, \cite{Radulescu.FundamentalGroup}, or more easily since $\mathbb N\subset\Fund{\VN{\FG_\infty}}$, \cite[Corollary 5.2.3]{Dykema.FreeBook}, it follows that $\Gen{\VN{\FG_\infty}}$ is either $0$ or $\infty$, by Corollary \ref{Scaling.FundamentalGroup}. Suppose that $\alpha=0$ so that $\Gen{\VN{\FG_k}}=0$ for each $k\geq 2$.  Consideration of the chain 
$$
\VN{\FG_2}\subset\VN{\FG_3}\subset\VN{\FG_4}\subset\dots\subset \VN{\FG_\infty}
$$
gives $\Gen{\VN{\FG_\infty}}=0$ by Shen's main technical theorem, \cite[Theorem 5.1]{Shen.Generators}.  Conversely, if $\Gen{\VN{\FG_\infty}}=0$, then the relation $\alpha=\Gen{\VN{\FG_2}}=0$ follows from the isomorphism
$$
\FG_2\cong\FG_\infty\rtimes\mathbb Z.
$$
This is essentially in \cite{Shen.Generators} and a proof can also be found in \cite[Chapter 15]{Sinclair.MasaBook}. Hence $\Gen{\VN{\FG_\infty}}=0$ if and only if $\alpha=0$ and the result follows. 
\end{proof}

As an immediate consequence, there is a direct link between the free group isomorphism problem and the generator invariant.

\begin{corollary}\label{FreeGroups.Cor}
If $\Gen{\VN{\FG_2}}>0$, then all the interpolated free group factors $\VN{\FG_r}$ are pairwise non-isomorphic for $r>1$.
\end{corollary}

\begin{remark}
Since $\VN{\FG_1}\cong L^\infty[0,1]$ has generator invariant $0$, the previous theorem extends to include the case $r=0$.
\end{remark}

The next results give some estimates regarding the generator invariant and free products.  Reverse inequalities to either Theorem \ref{FreeProduct.General} or Theorem \ref{FreeProduct.Hyperfinite} would immediately combine with Corollary \ref{FreeGroups.Cor} to show the non--isomorphism of the free group factors. Lemma \ref{FreeProduct.AmalDiffuse}  originates in a conjugacy idea used repeatedly by Shen to obtain \cite[Theorem 5.1]{Shen.Generators}.
\begin{lemma}\label{FreeProduct.AmalDiffuse}
Let $M$ and $N$ be \IIi factors containing a common diffuse von Neumann subalgebra $B$.  Then 
$$
\Gen{M*_BN}\leq\Gen{M}+\Gen{N}.
$$
\end{lemma}
\begin{proof}
Given $\eps>0$, choose finite subsets $X\subset M$ and $Y\subset N$ and families $P\in\P(M)$ and $Q\in\P(N)$ with $W^*(X)=M$, $W^*(Y)=N$ and 
\begin{equation}\label{FreeProduct.AmalDiffuse.1}
\I(X;P)<\Gen{M}+\frac{\eps}{2},\quad\I(Y;Q)<\Gen{N}+\frac{\eps}{2}.
\end{equation}
By refining if necessary, it may be assumed that $P$ and $Q$ are equivalent since Proposition \ref{Prelim.Refine} ensures that the estimate (\ref{FreeProduct.AmalDiffuse.1}) is unaffected by refinement.  Since $B$ is diffuse, choose $E\in\P_B$ equivalent to $P$ and $Q$ and unitaries $u\in M$ and $v\in N$ with $uPu^*=E$ and $vQv^*=E$.  Then $uXu^*\cup vYv^*$ is a finite generating set for $M*_BN$ and
$$
\I(uXu^*\cup vYv^*;E)=\I(X;P)+\I(Y;Q)<\Gen{M}+\Gen{N}+\eps,
$$
from which the result follows.
\end{proof}

\begin{theorem}\label{FreeProduct.General}
Let $M$ and $N$ be  \IIi factors.  Then
$$
\Gen{M*N}\leq \Gen{M}+\Gen{N}+\frac{1}{2}.
$$
\end{theorem}
\begin{proof}
Choose masas $A\subset M$ and $B\subset N$. Then $M\cong M*_AA$ and $N\cong N*_BB$, and so
$$
M*N\cong (M*_AA)*(B*_BN)\cong (M*_A\VN{\FG_2})*_BN,
$$
where $\VN{\FG_2}=A*B$.  Now use Lemma \ref{FreeProduct.AmalDiffuse} twice to obtain
$$
\Gen{M*N}\leq\Gen{M*_A\VN{\FG_2}}+\Gen{N}\leq\Gen{M}+\Gen{N}+\Gen{\VN{\FG_2}}.
$$
The result follows as Theorem \ref{Scaling.FreeGroup} gives $\Gen{\VN{\FG_2}}\leq 1/2$.
\end{proof}

The remainder of this section examines free products with finite hyperfinite von Neumann algebras.
In \cite{Jung.FreeEntropyHyperfinite}, Jung proves that
for a fixed hyperfinite von Neumann algebra
$Q$ with a fixed faithful normal tracial state $\phi$, 
Voiculescu's modified free entropy dimension $\delta_0(X)$ is the same for
all finite sets $X$ that generate $Q$.
Write $\delta_0(Q)$ for this quantity.
The definition of the modified free entropy dimension will be given in the next section, which discusses free entropy dimension in conjunction with the generator invariant.  Here  only the value of $\delta_0(Q)$ is needed.  Following \cite{Jung.FreeEntropyHyperfinite}, given $(Q,\phi)$, decompose $Q$ over its centre to obtain
$$
Q\cong Q_0\oplus\left(\bigoplus_{i=1}^s\Mat{k_i}\right),
$$
where $Q_0$ is diffuse or $\{0\}$, the sum on the right is either empty, finite or countably infinite and each $k_i\in\mathbb N$.  The trace $\phi$ is given by
$$
\phi=\alpha_0\phi_0\oplus\left(\bigoplus_{i=1}^s\alpha_i\ttr_{k_i}\right),
$$
where 
\begin{itemize}
\item $\alpha_0>0$ and $\phi_0$ is a faithful normal trace on $Q_0$, if $Q_0\neq\{0\}$;
\item $\alpha_0=0$ and $\phi_0=0$ if $Q_0=\{0\}$;
\item $\ttr_{k_i}$ is the tracial state on the ${k_i}\times{k_i}$ matrices $\Mat{k_i}$ and each $\alpha_i>0$.
\end{itemize}
Then, from \cite{Jung.FreeEntropyHyperfinite},
\begin{equation}\label{FreeProduct.FreeEntropy}
\delta_0(Q)=1-\sum_{i=1}^s\frac{\alpha_i^2}{k_i^2}.
\end{equation}
Furthermore, as also  noted in \cite{Jung.FreeEntropyHyperfinite}, this quantity agrees with the `free dimension number' for $Q$ defined in earlier work of the first author \cite{Dykema.FreeProductHyperfinite}.  In this work it was shown (Theorem 4.6 of~\cite{Dykema.FreeProductHyperfinite})
that if $A=L^\infty[0,1]$ is equipped with the usual trace $\int_0^1\cdot\ \mathrm{d}t$, then $A*Q\cong\VN{\FG_r}$, where $r=\delta_0(Q)+1$.

\begin{theorem}\label{FreeProduct.Hyperfinite}
Let $M$ be a \IIi factor and $Q$ a hyperfinite von Neumann algebra with a fixed faithful normalised trace $\phi$.  Then
$$
\Gen{M*Q}\leq\Gen{M}+\frac{1}{2}\delta_0(Q).
$$
\end{theorem}
\begin{proof}
Choose a masa $A\subset M$ so that $A$ is isomorphic to $L^\infty[0,1]$ with the trace on $A$ (coming from $\tr_M$) being given by $\int_0^1\cdot\ \mathrm{d}t$.  The discussion preceding the theorem gives
$$
M*Q\cong M*_A*(A*Q)\cong M*_A\VN{\FG_r}
$$
where $r=1+\delta_0(Q)$ and $A$ is a masa in $\VN{\FG_r}$.  By Proposition \ref{FreeProduct.AmalDiffuse} and Theorem \ref{Scaling.FreeGroup},
$$
\Gen{M*Q}\leq\Gen{M}+\Gen{\VN{\FG_r}}\leq\Gen{M}+\frac{r-1}{2}=\Gen{M}+\frac{1}{2}\delta_0(Q),
$$
exactly as required.
\end{proof}

The next two corollaries are obtained by taking $Q$ to be successively the $n\times n$ matrices with the usual normalised trace and to be $\VN{\mathbb Z_n}$, with the group trace.  The results follow from calculating $\delta_0(\Mat{n})=1-\frac{1}{n^2}$ and $\delta_0(\VN{\mathbb Z_n})=1-\frac{1}{n}$ from Jung's formula (\ref{FreeProduct.FreeEntropy}).
\begin{corollary}\label{FreeProduct.Mn}
Let $M$ be a \IIi factor and $n\geq 2$.  Then 
$$
\Gen{M*\Mat{n}}\leq\Gen{M}+\frac{1}{2}-\frac{1}{2n^2}.
$$
\end{corollary}
\begin{corollary}\label{FreeProduct.Zn}
Let $M$ be a \IIi factor and $n\geq 2$. Then
$$
\Gen{M*\VN{\mathbb Z_n}}\leq\Gen{M}+\frac{1}{2}-\frac{1}{2n}.
$$
\end{corollary}

\section{Free entropy dimension}
The objective in this section is to relate the generator invariant to Voiculescu's modified
free entropy dimension by proving the inequalities
\begin{equation}\label{Entropy.NonSA}
\delta_0(X)\leq 1+2\I(X),
\end{equation}
when $X$ is a finite generating set in a finite von Neumann algebra $M$ and,
under the extra assumption that $X$ consists of self--adjoint elements,
\begin{equation}\label{Entropy.SA}
\delta_0(X)\leq 1+\I(X).
\end{equation}
For certain sets $X$ of operators, these inequalities give lower bounds on $\I(X)$.
These seem to be the only such lower bounds that are currently known.
Consider the case of a $DT$-operator $Z$,
introduced by the first author and Haagerup in \cite{Dykema.DT}.
By~\cite{Dykema.DTSubspaces} each such $Z$, including the quasi--nilpotent DT--operator $T$, generates $\VN{\FG_2}$.
The operator $Z$ is constructed by realising $\VN{\FG_2}$ as generated by a 
semicircular element $S$ together with a free copy of $L^\infty[0,1]$, using projections
from $L^\infty[0,1]$ to cut out an ``upper triangular'' part of $S$, which is the quasi--nilpotent
DT--operator $T$, and then adding an operator from $L^\infty[0,1]$ to get $Z$.
Using projections from $L^\infty[0,1]$, one easily sees that
$\I(Z)\leq1/2$.
On the other hand, since, by \cite{Dykema.DTEntropy}, $\delta_0(Z)=2$,
(\ref{Entropy.NonSA}) gives
\[
\I(Z)=1/2.
\]
Similarly, free semi--circular elements $h_1,h_2$ generating $\VN{\FG_2}$ satisfy $\delta_0(h_1,h_2)=2$ from \cite{Voiculescu.Entropy2,Voiculescu.Entropy3}, so that (\ref{Entropy.SA}) gives $\I(h_1,h_2)\geq 1$.  The reverse inequality follows by taking a sufficiently fine family of projections in either $W^*(h_1)$ or $W^*(h_2)$ --- see the proof of Lemma \ref{SelfAd.NGenerators} --- so that
\[
\I(h_1,h_2)=1.
\]

The central question in the theory of microstates free entropy dimension is that of invariance: if one has two finite generating sets $X_1$ and $X_2$ for the same \IIi factor must $\delta_0(X_1)=\delta_0(X_2)$?  Certain conditions on the factor, such as having a Cartan masa (\cite{Voiculescu.Entropy3}) and non-primeness (\cite{Ge.AppFreeEntropy}), imply that all finite generating sets have free-entropy at most one.  Independently Jung \cite{Jung.Strong1Bounded} and Hadwin and Shen \cite{Hadwin.OrbitDimension} have shown that the existance of a finite generating set $X$ for a \IIi factor with $\delta_0(X)=1$ and certain additional technical conditions, which are subtly different in each approach, imply that all finite generating sets $Y$ for $M$ have $\delta_0(Y)\leq 1$.  Before discussing these conditions further and establishing inequalities (\ref{Entropy.NonSA}) and (\ref{Entropy.SA}), the prevailing notation in this area is outlined.  The definition of the modified free-entropy dimension given below is Jung's covering--number reformulation from \cite{Jung.FreeEntropyLemma,Jung.HyperfiniteInequality}.

Let $M$ be a \IIi factor and continue to write $\tr$ for the faithful normal trace on $M$ normalised by $\tr(1)=1$.  For $k\in\mathbb N$ write $\Mat{k}$ for the $k\times k$ matrices, equipped with the trace $\TR_k$ normalised with $\TR_k(1)=1$. For a finite subset $X=\{x_1,\dots,x_n\}\subset M$, $\gamma>0$, $k,m\in\mathbb N$ and $R>0$ define the microstate space $\Gamma_R(X;m,k,\gamma)$ to be the set of all $n$-tuples $(a_1,\dots,a_n)$ of $k\times k$-matrices whose $^*$-moments approximate those of $(x_1,\dots,x_n)$ up to order $m$ within a tolerance of $\gamma$ and whose norms are bounded by $R$.
This means that $\nm{a_i}\leq R$ for all $i$ and
$$
\left|\tr(x_{i_1}^{j_1}\dots x_{i_p}^{j_p})-\TR_k(a_{i_1}^{j_1}\dots a_{i_p}^{j_p})\right|<\gamma,
$$
for all $p\leq m$, $i_1,\dots,i_p\in\{1,\dots,n\}$ and $j_1,\dots,j_m\in\{1,*\}$. When all the $x_j$'s are self--adjoint, it makes no difference to the definition of $\delta_0(X)$ whether  all the $a_i$'s are required to be self--adjoint --- see for example the beginning
of Section 3 of~\cite{Dykema.DTEntropy}.

Given $\eps>0$, the covering number $\Cover{\eps}{Y}$ of a metric space $Y$  is the minimal cardinality of an $\eps$-net for $Y$.  One easy estimate,  used in the sequel, is
\begin{equation}\label{Entropy.CoverPack}
\Cover{\eps}{Y}\leq\Pack{\eps/2}{Y},
\end{equation}
where $\Pack{\eps/2}{Y}$ is the maximal number of disjoint open $\eps/2$-balls which can be found in $Y$.  In \cite{Jung.FreeEntropyLemma}, Jung defines, for $m\in\mathbb N$, $\gamma,\eps>0$ and $R>0$,
\begin{equation}\label{Entropy.10}
\mathbb K_{\eps,R}(X;m,\gamma)=\limsup_{k\rightarrow\infty}k^{-2}\log\Cover{\eps}{\Gamma_R(X;m,k,\gamma)}.
\end{equation}
where the metric on $\Gamma_R(X;m,k,\gamma)$ is that obtained from the Euclidian norm on $(\Mat{k})^n$  given by
\begin{equation}\label{Entropy.DefMetric}
\nm{(a_1,\dots,a_n)}_2^2=\left(\sum_{l=1}^n\TR_k(a_l^*a_l)\right)^{1/2}.
\end{equation}
Then define
\begin{equation}\label{Entropy.11}
\mathbb K_{\eps,R}(X)=\inf_{m\in\mathbb N,\gamma>0}\mathbb K_{\eps,R}(X;m,\gamma).
\end{equation}
and
$$
\mathbb K_\eps(X)=\sup_{R>0}K_{\eps,R}(X).
$$
Jung shows in Corollary 24 of \cite{Jung.FreeEntropyLemma}, that the modified free entropy dimension $\delta_0(X)$ is given by
\begin{equation}\label{Entropy.DefDelta}
\delta_0(X)=\limsup_{\eps\rightarrow 0}\frac{\mathbb K_{\eps}(X)}{|\log \eps|}.
\end{equation}

In \cite{Jung.Strong1Bounded}, Jung made a detailed analysis of the rate of convergence
in the limit-superior in (\ref{Entropy.DefDelta}). For $\alpha>0$, say that
the generating set $X$ of $M$ to be \emph{$\alpha$-bounded}
if, for some $R\geq\max_{x\in X}\nm{x}$, there exists some constant $K$ and $\eps_0>0$ such that
$$
\mathbb K_{\eps,R}(X)\leq \alpha|\log\eps|+K,
$$
for all $0<\eps<\eps_0$.  Corollary 1.4 of \cite{Jung.Strong1Bounded} demonstrates that this is equivalent to the original definition of $\alpha$-boundedness in \cite{Jung.Strong1Bounded} and that if $X$ is $\alpha$-bounded, then $\delta_0(X)\leq\alpha$.
Furthermore, $M$ is said to be \emph{strongly--$1$--bounded} if it has a generating set $X$
that is $1$-bounded and if there exists a self-adjoint element $x$ belonging to the $^*$-algebra
generated by $X$  that has finite free entropy.
Theorem 3.2 of \cite{Jung.Strong1Bounded} shows that if
$M$ is strongly--$1$--bounded, then every finite set of generators for $M$ is $1$--bounded.

The alternative approach of Hadwin and Shen in \cite{Hadwin.OrbitDimension} is to examine coverings by unitary orbits of balls. For $m,k,\gamma,R$ as above and $\eps>0$, define 
$$
\nu_\eps(\Gamma_R(x_1,\dots,x_n;m,k,\gamma))
$$
to be the minimum cardinality of a set $\Lambda\subset \Gamma_R(x_1,\dots,x_n;m,k,\gamma)$ such that for every $a=(a_1,\dots,a_n)\in\Gamma_R(x_1,\dots,x_n;m,k,\gamma)$ there is some unitary $u\in\Mat{k}$ and $b\in\Lambda$ with $\nm{uau^*-b}_2\leq\eps$, i.e. the microstate space $\Gamma_R(x_1,\dots,x_n;m,k,\gamma)$ is covered by $\nu_\eps(\Gamma_R(x_1,\dots,x_n;m,k,\gamma))$ orbit $\eps$-balls.  Here $uau^*$ is defined to be the $n$-tuple $(ua_1u^*,\dots,ua_nu^*)$.  Then define
$$
\ko(x_1,\dots,x_n;R,\eps)=\inf_{m\in\mathbb N,\gamma>0}\limsup_{k\rightarrow\infty}\frac{\log(\nu_\eps(\Gamma_R(x_1,\dots,x_n;m,k,\gamma)))}{-k^2\log\eps}
$$
and the \emph{upper free-orbit dimension} $\ko_2(x_1,\dots,x_n)$ by
$$
\ko_2(x_1,\dots,x_n)=\sup_{0<\eps<1}\sup_{R>0}\ko(x_1,\dots,x_n;R,\eps).
$$
The main theorem of \cite{Hadwin.OrbitDimension} (Section 3, Theorem 1) is that if $X$ is a finite set of generators for a \IIi factor with $\ko_2(X)=0$, then every finite set of generators $Y$ for this factor also has $\ko_2(Y)=0$. Therefore, it makes sense to say a \IIi factor has zero upper free orbit-dimension when it has a finite generating set $X$ with $\ko_2(X)=0$.

The connection between the upper free orbit-dimension and Voiculescu's modified free entropy dimension is immediate from results of Szarek \cite{Szarek.NetsGrassmann} which show that there is an absolute constant $C>0$ such that the groups $\Unitary_k$ of unitary matrices in $\Mat{k}$ equipped with the metric induced from the operator norm satisfy
$$
K_\eps(\Unitary(K))\leq \left(\frac{C}{\eps}\right)^{k^2}.
$$
Given a finite generating set $X=(x_1,\dots,x_n)$ for a \IIi factor $M$, take $S=2(\sum_{i=1}^n\nm{x_i}_2^2)^{1/2}$. Provided $m\geq 2$ and $\gamma$ is sufficiently small, the estimate
\begin{align*}
K_\eps(\Gamma_R(X;m,k,\gamma))&\leq\nu_{\eps/2}(\Gamma_R(X;m,k,\gamma))\cdot K_{\eps/8S}(\Unitary_k)\\
&\leq\left(\frac {8CS}{\eps}\right)^{k^2}\nu_{\eps/2}(\Gamma_R(X;m,k,\gamma)),
\end{align*}
follows. Therefore
\begin{align*}
\mathbb K_{\eps,R}(X;m,\gamma)&\leq\log(8CS/\eps)+\limsup_{k\rightarrow\infty}\log\nu_{\eps/2}(\Gamma_R(X;m,k,\gamma))
\end{align*}
and
\begin{align}
\mathbb K_{\eps,R}(X)&\leq\log(8CS/\eps)+\ko(X;R,\eps/2)|\log(\eps/2)\nonumber|\\
&\leq\left(1+\ko(X;R,\eps/2)\right)|\log\eps|+\ko(X;R,\eps/2)\log(2)+\log(8CS)\label{Entropy.FreeOrbit.1}.
\end{align}
The free orbit-dimension $\ko_1(X)$ of \cite{Hadwin.OrbitDimension} is defined by
$$
\ko_1(X)=\limsup_{\eps\rightarrow 0}\sup_{R>0}\ko(X;R,\eps).
$$
The relationships between the approaches of \cite{Jung.Strong1Bounded} and \cite{Hadwin.OrbitDimension} follow from the definition of $\ko_1$ and (\ref{Entropy.FreeOrbit.1}).
\begin{proposition}
Let $X$ be a finite set of generators for a \IIi factor $M$. Then $\delta_0(X)\leq 1+\ko_1(X)$ and $X$ is $\alpha$-bounded, for all $\alpha>1+\ko_1(X)$. 
\end{proposition}

\begin{proposition}
A \IIi factor whose upper free orbit dimension is zero, i.e. has a finite generating set $X$ with $\ko_2(X)=0$, is strongly $1$-bounded.
\end{proposition}
\begin{proof}
By hypothesis there is a set of self-adjoint generators $X$ for the \IIi factor with $\ko_2(X)=0$.  If there is no element of finite entropy in the $^*$-algebra generated by $X$, adjoin one. This new set of generators $Y$ also has zero upper free orbit-dimension by \cite[Section 3, Theorem 1]{Hadwin.OrbitDimension}. The definition of $\ko_2$ ensures that
$$
\mathbb K_{\eps,R}(Y)\leq|\log\eps|+\log(8CS),
$$
from equation (\ref{Entropy.FreeOrbit.1}).  That is $Y$ is $1$-bounded, and so $M$ is strongly $1$-bounded.
\end{proof}

In this section the upper free-orbit dimension is used as it enables us to show, in Corollary \ref{Entropy.FreeGroup}, that no set of generators $X$ for an interpolated free group factor can have $\I(X)=0$. The strongly-$1$-bounded approach to this would only rule out the existance of such generators $X$ which also have an element of finite entopy in the $^*$-algebra they generate. Before proceeding to the main result of this section, recall the following lemma, due to Voiculescu. Formally, a proof can be constructed by copying the ideas of 
the proof of Proposition 1.6 of~\cite{Voiculescu.Entropy3}.
\begin{lemma}\label{Entropy.Proj}
Suppose that $X=\{x_1,\dots,x_n\}$ is a finite subset of $M$ with $W^*(X)=M$.
Fix $p\in\mathbb N$ and pairwise orthogonal projections $e_1,e_2,\dots,e_p$ in $M$ with sum $1$
and $\tau(e_i)=p^{-1}$ for each $i$.
Given $\gamma>0$ and $m\in\mathbb N$, there exists $\gamma'>0$, $m',k'\in\mathbb N$ such that if 
$$
(a_1,\dots,a_n)\in \Gamma(X;m',k,\gamma'),
$$
for some $k\geq k'$ with $p|k$, then there exist pairwise orthogonal projections $(f_1,\dots,f_p)$ in $\Mat{k}$ each with $\TR_k(f_i)=p^{-1}$ (so that $\sum_{i=1}^nf_i=1$) satisfying
$$
(a_1,\dots,a_n,f_1,\dots,f_p)\in\Gamma(x_1,\dots,x_n,e_1,\dots,e_p;m,k,\gamma).
$$
\end{lemma}

Now the main result of this section relating $\I(X)$ to $\ko_1(X)$ and $\ko_2(X)$.
\begin{proposition}\label{Entropy.Boundedness}
Let $M$ be a finite von Neumann algebra and $X\subset M$ a finite generating set for $M$. In general, $\ko_1(X)\leq2\I(X)$ and in the case that each $x\in X$ is self-adjoint, $\ko_1(X)\leq\I(X)$. Furthermore, if $\I(X)=0$, then $\ko_2(X)=0$.
\end{proposition}
\begin{proof}
Let $X=(x_1,\dots,x_n)$ be an $n$-tuple which generates $M$.  Fix $\eps>0$ and suppose that $c>\I(X)$. Use Lemma \ref{Prelim.Diagonalisable} to find $E=\{e_1,\dots,e_p\}\in\Peq(M)$ with $\I(X;E)<c$ and each $\tau(e_i)=p^{-1}$.  Let $R>0$ and write $S=(\sum_{i=1}^n\nm{x_i}_2^2)^{1/2}$.  Take $\gamma>0$ and $m\in\mathbb N$ with $m\geq 6$. Let $k',m'$ and $\gamma'$ be the constants obtained by applying Lemma \ref{Entropy.Proj} to $X$ and $(e_1,\dots,e_p)$.  Let $k\geq k'$ be divisible by $p$ and fix a family of pairwise orthogonal projections $(q_1,\dots,q_p)$ in $\Mat{k}$ with each $\tr_k(q_i)=p^{-1}$.  

For each $l=1,\dots,n$ write $T_l$ for the set of pairs $(i,j)$ with $e_ix_le_j\neq 0$.  As $\I(I;E)<c$, it follows that 
\begin{equation}\label{Entropy.SizeT}
\sum_{l=1}^n|T_l|<cp^2.
\end{equation}
Following the approach of \cite{Ge.AppFreeEntropy}, define the projection $Q$ from $(\Mat{k})^n$ into $(\Mat{k})^n$ by
$$
Q(a_1,\dots,a_n)=\left(\sum_{(i,j)\in T_l}q_ia_lq_j\right)_{l=1}^n.
$$
When each $x_l=x_l^*$, the sets $T_l$ are invariant under the adjoint operation so in this case $Q$ restricts to give a projection from $(\MatSA{k})^n$ into $(\MatSA{k})^n$. The range $Q((\Mat{k})^n)$ is a $2\sum_{l=1}^n|T_l|(k/p)^2$--dimensional subspace of $(\Mat{k})^n\cong\mathbb R^{2nk^2}$.
Under the additional assumption that $x_l=x_l^*$ for each $l$,
the range $Q((\MatSA{k})^n)$ is a $\sum_{l=1}^n|T_l|(k/p)^2$--dimensional subspace of
$(\MatSA{k})^n\cong\mathbb R^{nk^2}$.  

Let $Y$ be the subset $(Q(\Mat{k}^n))_{2S}$ consisting of all elements $(b_1,\dots,b_n)\in Q(\Mat{k}^n)$ with $\nm{(b_1,\dots,b_n)}_2\leq 2S$ if $X$ does not consist of self-adjoint elements, and let $Y=(Q((\MatSA{k})^n))_{2S}$ when each $x_l=x_l^*$.  Volume considerations give
$$
\Pack{\eps/2}{Y}\leq\left(\frac{2S}{\eps/2}\right)^{2\sum_{l=1}^n|T_l|(k/p)^2},
$$
in the first case and
$$
\Pack{\eps/2}{Y}\leq\left(\frac{2S}{\eps/2}\right)^{\sum_{l=1}^n|T_l|(k/p)^2},
$$
in the second. The simple estimate (\ref{Entropy.CoverPack}) combines with (\ref{Entropy.SizeT}) to give
$$
\Cover{\eps}{Y}\leq\begin{cases}\left(\frac{4S}{\eps}\right)^{ck^2},&x_l=x_l^*\textrm{ for all } l\\
\left(\frac{4S}{\eps}\right)^{2ck^2},&\textrm{otherwise.}\end{cases}
$$

Given any point $a=(a_1,\dots,a_n)\in\Gamma_R(X;m',k,\gamma)$, let $(f_1,\dots,f_p)$ be the orthogonal projections in $\Mat{k}$ with each $\tau_k(f_i)=p^{-1}$ from Lemma \ref{Entropy.Proj}. Now each $x_l$ can be written $x_l=\sum_{(i,j)\in T_l}e_ix_le_j$. Since
$$
(a_1,\dots,a_n,f_1,\dots,f_p)\in\Gamma_R(x_1,\dots,x_n,e_1,\dots,e_p;m,k,\gamma),
$$
it follows that
$$
\nm{a_l-\sum_{(i,j)\in T_l}f_ia_lf_j}^2_2<\gamma,
$$
from the standing assumption that $m\geq 6$.  Take a unitary $u\in\Mat{k}$ such that $uf_iu^*=q_i$ for each $i=1,\dots,p$.  Then $u^*Q(uau^*)u$ is the $n$-tuple with $\sum_{(i,j)\in T_l}f_ia_lf_j$ in the $l$-entry.  Thus
$$
\nm{a-u^*Q(uau^*)u}_2^2\leq n\gamma.
$$
Provided $\gamma$ is taken small enough, $\nm{a}_2\leq 2S$, so that $Q(uau^*)$ lies in the ball $Y$. If in addition, $(n\gamma)^{1/2}\leq\eps/2$, it follows at most $K_{\eps/2}(Y)$ orbit $\eps$-balls are required to cover $\Gamma_R(x_1,\dots,x_n;m,k,\gamma)$. That is
$$
\nu_\eps(\Gamma_R(x_1,\dots,x_n;m,k,\gamma)\leq\begin{cases}\left(\frac{8S}{\eps}\right)^{ck^2},&x_l=x_l^*\text{ for all }l;\\\left(\frac{8S}{\eps}\right)^{2ck^2},&\text{otherwise},\end{cases}
$$
for $R>0$, $m\geq 6$, $\gamma$ sufficiently small and $k\geq k'$ which is divisible by $p$.  Thus
\begin{align*}
\ko(x_1,\dots,x_n;R,\eps)&=\inf_{m\in\mathbb N,\gamma>0}\limsup_{k\rightarrow\infty}\frac{\log\nu_\eps(\Gamma_R(x_1,\dots,x_n;m,k,\gamma))}{-k^2\log\eps}\\
&\leq \begin{cases}c\left(1+\frac{\log(8S)}{|\log\eps|}\right),&x_l=x_l^*,\text {for all }l,\\2c\left(1+\frac{\log(8S)}{|\log\eps|}\right),&\text{otherwise}.\end{cases}
\end{align*}
This holds for each $c>\I(X)$ and so we can replace $c$ by $\I(X)$ in the inequality above to obtain
\begin{align*}
\ko_1(x_1,\dots,x_n)&=\limsup_{\eps\rightarrow 0}\sup_{R>0}\ko(x_1,\dots,x_n;R,\eps)\\
&\leq\begin{cases}\I(X),&x_l=x_l^*,\text{ for all }l,\\2\I(X),&\text{otherwse}.\end{cases}
\end{align*}
in general, and if $\I(X)=0$, then $\ko_2(X)=0$.
\end{proof}

The first two corollaries below follow immediately.
\begin{corollary}
Let $X$ be a finite set of generators for the \IIi factor $(M,\tau)$. Then
$$
\delta_0(X)\leq 1+2\I(X).
$$
\end{corollary}
\begin{corollary}
Let $X$ be a finite set of self-adjoint generators for the \IIi factor $(M,\tau)$. Then
$$
\delta_0(X)\leq 1+\I(X).
$$
\end{corollary}

The next corollary follows from Proposition \ref{Entropy.Boundedness} and \cite[Section 3, Theorem 1]{Hadwin.OrbitDimension}. It shows in particular that $\I(X)>0$ for every generating set $X$ of an interpolated free group factor by \cite{Voiculescu.Entropy2,Voiculescu.Entropy3}.
\begin{corollary}\label{Entropy.FreeGroup}
Let $M$ be any \IIi factor having a finite generating set $Y$ with $\delta_0(Y)>1$. If $X$ is any finite generating set for $M$, then $\I(X)>0$.
\end{corollary}

\begin{tabular*}{\textwidth}{l@{\hspace*{2cm}}l}
Ken Dykema \& Roger Smith               &Allan Sinclair\\
Department of Mathematics &School of Mathematics\\
Texas A\&M University     &University of Edinburgh\\
College Station, TX 77843           &Edinburgh EH9 3JZ\\
USA           &UK\\
 \texttt{kdyekma@math.tamu.edu}& \texttt{a.sinclair@ed.ac.uk}\\
\texttt{rsmith@math.tamu.edu} \\
{}\\
Stuart White\\
Department of Mathematics\\
University of Glasgow\\
University Gardens\\
Glasgow G12 8QW\\
UK\\
\texttt{s.white@maths.gla.ac.uk}
\end{tabular*}

\end{document}